\pdfoutput=1
\documentclass[1p,fleqn,preprint]{elsarticle}
\usepackage{amsmath,amssymb,amsthm}
\usepackage{mathtools,mathrsfs}
\usepackage{relsize}
\usepackage{ifpdf}

\ifpdf
\usepackage[bookmarks=false,%
pdfstartview=FitH,linkbordercolor={0.5 1 1},%
citebordercolor={0.5 1 0.5},unicode,%
hyperindex]{hyperref}%
\else
\usepackage{breakurl}                          
\fi

\biboptions{numbers,sort&compress}

\newtheorem{theorem}{Theorem}
\newtheorem{lemma}{Lemma}
\newtheorem{example}{Example}
\newtheorem{corollary}{Corollary}
\newtheorem*{remark*}{Remark}
\newtheorem*{question*}{Question}

\newcommand{\bbR}{\mathbb{R}}
\newcommand{\setA}{\mathscr{A}}
\newcommand{\setB}{\mathscr{B}}
\newcommand{\Hset}{\mathcal{H}}
\newcommand{\IRU}{\mathcal{U}}
\newcommand{\Lin}{\mathcal{L}}
\newcommand{\Kset}{\mathcal{K}}
\newcommand{\Mset}{\mathcal{M}}
\newcommand{\1}{\mathbf{1}}

\DeclareMathOperator*{\co}{co}
\newcommand{\transpose}{\mathsmaller{\mathsf{T}}}

\let\ge\geqslant
\let\le\leqslant

\begin{document}
\begin{frontmatter}
\title{Hourglass alternative and the finiteness conjecture for the spectral
characteristics of sets of non-negative matrices\tnoteref{rfs}}

\tnotetext[rfs]{Funded by the Russian Science Foundation, Project No.
14-50-00150.}

\author{Victor Kozyakin}

\address{Institute
for Information Transmission Problems\\ Russian Academy of Sciences\\ Bolshoj
Karetny lane 19, Moscow 127994 GSP-4, Russia}

\ead{kozyakin@iitp.ru}
\ead[url]{http://www.iitp.ru/en/users/46.htm}

\begin{abstract}

Recently Blondel, Nesterov and Protasov
proved~\cite{BN:SIAMJMAA09,NesPro:SIAMJMAA13} that the finiteness
conjecture holds for the generalized and the lower spectral radii of the
sets of non-negative matrices with independent row/column uncertainty. We
show that this result can be obtained as a simple consequence of the
so-called hourglass alternative used in~\cite{ACDDHK15},  by the author and
his companions, to analyze the minimax relations between the spectral radii
of matrix products. Axiomatization of the statements that constitute the
hourglass alternative makes it possible to define a new class of sets of
positive matrices having the finiteness property, which includes the sets
of non-negative matrices with independent row uncertainty. This class of
matrices, supplemented by the zero and identity matrices, forms a semiring
with the Minkowski operations of addition and multiplication of matrix
sets, which gives means to construct new sets of non-negative matrices
possessing the finiteness property for the generalized and the lower
spectral radii.
\end{abstract}


\end{frontmatter}

\section{Introduction}\label{S:intro}
One of the characteristics that describes the exponential growth rate of the
matrix products with factors from a set of matrices, is the so-called joint
or generalized spectral radius. Denote by $\Mset(N,M)$ the set of all real
$(N\times M)$-matrices. This set of matrices is naturally identified with the
space $\bbR^{N\times M}$ and therefore, depending on the context, it can be
interpreted as topological, metric or normed space. The \emph{joint spectral
radius}~\cite{RotaStr:IM60} of a set of matrices $\setA\subset\Mset(N,N)$ is
defined as the value of
\begin{equation}\label{E-GSRad0}
\rho(\setA)=
\adjustlimits\lim_{n\to\infty}\sup_{A_{i}\in\setA}\|A_{n}\cdots A_{1}\|^{1/n}
=\adjustlimits\inf_{n\ge 1}\sup_{A_{i}\in\setA}\|A_{n}\cdots A_{1}\|^{1/n},
\end{equation}
where $\|\cdot\|$ is a matrix norm on $\Mset(N,M)$ generated by the
corresponding vector norm on $\bbR^{N}$. The \emph{generalized spectral
radius}~\cite{DaubLag:LAA92,DaubLag:LAA01} of a bounded set of matrices
$\setA\subset\Mset(N,N)$ is the quantity
\begin{equation}\label{E-GSRad}
\hat{\rho}(\setA)=
\adjustlimits\lim_{n\to\infty}\sup_{A_{i}\in\setA}\rho(A_{n}\cdots A_{1})^{1/n}
=\adjustlimits\sup_{n\ge 1}\sup_{A_{i}\in\setA}\rho(A_{n}\cdots A_{1})^{1/n},
\end{equation}
where $\rho(\cdot)$ is the spectral radius of a matrix, i.e. the maximum of
modules of its eigenvalues. If the norms of matrices from the set $\setA$ are
uniformly bounded then by the Berger-Wang theorem~\cite{BerWang:LAA92}
$\rho(\setA)=\hat{\rho}(\setA)$. In the case when the set of matrices $\setA$
is compact (closed and bounded), the suprema over $A_{i}\in\setA$ in
\eqref{E-GSRad0} and \eqref{E-GSRad} may be replaced by maxima.

By replacing the suprema over $A_{i}\in\setA$ in \eqref{E-GSRad} with infima
(or with minima, in the case of a compact set of matrices) one can obtain the
so-called \emph{joint spectral subradius} or \emph{lower spectral
radius}~\cite{Gurv:LAA95}:
\begin{equation}\label{E-LSRad0}
\check{\rho}(\setA)=
\adjustlimits\lim_{n\to\infty}\inf_{A_{i}\in\setA}\|A_{n}\cdots A_{1}\|^{1/n}
=\adjustlimits\inf_{n\ge 1}\inf_{A_{i}\in\setA}\|A_{n}\cdots A_{1}\|^{1/n},
\end{equation}
which also (for arbitrary, not necessarily bounded set of matrices) may be
expressed in terms of the spectral radii instead of norms:
\begin{equation}\label{E-LSRad}
\check{\rho}(\setA)=
\adjustlimits\lim_{n\to\infty}\inf_{A_{i}\in\setA}\rho(A_{n}\cdots A_{1})^{1/n}
=\adjustlimits\inf_{n\ge 1}\inf_{A_{i}\in\setA}\rho(A_{n}\cdots A_{1})^{1/n},
\end{equation}
as was shown in~\cite[Theorem~B1]{Gurv:LAA95} for finite sets $\setA$, and
in~\cite[Lemma~1.12]{Theys:PhD05} and~\cite[Theorem~1]{Czornik:LAA05} for
arbitrary sets $\setA$.

The possibility of explicit calculation of the spectral characteristics of
sets of matrices is conventionally associated with the validity of the
\emph{finiteness conjecture} according to which the limit in formulas
\eqref{E-GSRad} and \eqref{E-LSRad} is attained at some finite value of $n$.
For the generalized spectral radius this conjecture was set up by Lagarias
and Wang~\cite{LagWang:LAA95} and subsequently disproved by  Bousch and
Mairesse~\cite{BM:JAMS02}. Later on there appeared a few alternative
counterexamples~\cite{BTV:SIAMJMA03,Koz:CDC05:e,Koz:INFOPROC06:e}. However,
all these counterexamples were pure `non-existence' results which provided no
constructive description of the sets of matrices for which the finiteness
conjecture fails. The first explicit counterexample to the finiteness
conjecture was built in~\cite{HMST:AdvMath11}, while the general methods of
constructing of such a type of counterexamples were presented recently
in~\cite{MorSid:JEMS13,JenPoll:ArXiv15}. The lower radius in a sense is more
complex object for analysis than the generalized spectral radius because it
generally depends on $\setA$ not
continuously~\cite{Jungers:LAA12,BochiMor:ArXiv13}. However, for the lower
spectral radius, disproof of the finiteness conjecture was found to be even
easier~\cite{BM:JAMS02,CJ:IJAMCS07} than for the generalized spectral radius.

Despite the fact that in general the finiteness conjecture is false, attempts
to discover new classes of matrices for which it still occurs continues.
However, it should be borne in mind that the validity of the finiteness
conjecture for some class of matrices provides only a theoretical possibility
to explicitly calculate the related spectral characteristics, because in
practice calculation of the spectral radii $\rho(A_{n}\cdots A_{1})$ for all
possible sets of matrices $A_{1},\ldots, A_{n} \in\setA$ may require too much
computing resources, even for relatively small values of $n$. Therefore, from
a practical point of view, the most interesting are the cases when the
finiteness conjecture is valid for small values of $n$.

The finiteness conjecture is obviously satisfied for the sets of commuting
matrices as well as for sets consisting of upper or lower triangular matrices
or of matrices that are isometries in some norm up to a scalar factor (that
is, $\|Ax\|\equiv\lambda_{A}\|x\|$ for some $\lambda_{A}$). It is less
obvious that the finiteness conjecture holds for the class of `symmetric'
bounded sets of matrices characterizing by the property that together with
each matrix the corresponding set contains also the (complex) conjugate
matrix~\cite[Proposition~18]{PW:LAA08}. In particular, this class includes
all the sets of self-adjoint matrices. One of the most interesting classes of
matrices for which the finiteness conjecture is valid, for both the
generalized and the lower spectral radius, is the so-called class of
non-negative matrices with independent row uncertainty~\cite{BN:SIAMJMAA09}.
Note that in all these cases, the generalized spectral radius coincides with
the spectral radius of a single matrix from $\setA$ or with the spectral
radius of the product of a pair of such matrices.

In~\cite{JB:LAA08} it was demonstrated that the finiteness conjecture is
valid for any pair of $2\times 2$ binary matrices, i.e. matrices with the
elements $\{0,1\}$, and in~\cite{CGSZ:LAA10} a similar result was proved for
any pair of $2\times 2$ sign-matrices, i.e. matrices with the elements $\{-
1,0,1\}$. Finally,
in~\cite{DHLX:LAA12,LiuXiao:LNCS12,JM:LAA13,Morris:ArXiv11,WangWen:CIS13} it
was shown that the finiteness conjecture holds for any bounded family of
matrices $\setA$, whose matrices, except perhaps one, have rank~$1$. There
are also a number of works with less constructive sufficient conditions for
attainability of the generalized spectral radius on a finite product of
matrices, see, e.g., the references in~\cite{Koz:IITP13}.

So, calculating the joint and lower spectral radii is a challenging problem,
and only for exceptional classes of matrices these characteristics may be
found explicitly and expressed by a `closed formula', see,
e.g.,~\cite{Jungers:09,Koz:IITP13} and the bibliography therein.


Outline the structure of the work. In this section, we have presented a brief
overview of the results related to the finiteness conjecture for the spectral
characteristics of matrices. In Section~\ref{S:IRU}, we remind the definition
of the sets of non-negative matrices with independent row uncertainty, and
then in Theorem~\ref{T:BNP} give a new proof of the related
Blondel-Nesterov-Protasov results on finiteness~\cite{NesPro:SIAMJMAA13}. A
key point of this proof is the so-called hourglass alternative,
Lemma~\ref{L:alternative}, that has been proposed in~\cite{ACDDHK15} to
analyze the minimax relations between the spectral radii of matrix products.
In Section~\ref{S:Hsets}, assertions of Lemma~\ref{L:alternative} are taken
as axioms for the definition of the sets of positive matrices, called
hourglass- or $\Hset$-sets, satisfying the hourglass alternative. In
Theorem~\ref{T:semiring} we show that the totality of all $\Hset$-sets of
matrices, supplemented by the zero and the identity matrices, forms a
semiring. This opens up the possibility of constructing new classes of
matrices for which analogues of Theorem~\ref{T:BNP} are true. The main result
of such a kind, Theorem~\ref{T:Hset}, is proved in Section~\ref{S:main}, and
in Corollary~\ref{C1} we show that all the assertions on finiteness of the
spectral characteristics remain valid for the sets of matrices, obtained as a
polynomial Minkowski combination of compact sets of non-negative matrices
with independent row uncertainty. In Section~\ref{S:cosets}, we present a
general fact about the relationship between the spectral characteristics of
sets of matrices and their convex hulls. Concluding remarks are given in
Section~\ref{S:conclude}.

\section{Sets of matrices with independent row uncertainty}\label{S:IRU}
As was noted in Section~\ref{S:intro}, one of the most interesting classes of
matrices for which the finiteness conjecture holds, both for the generalized
and lower spectral radius, is the so-called class of non-negative matrices
with independent row uncertainty~\cite{BN:SIAMJMAA09}. In this section, we
recall the relevant definition and present a new proof of the corresponding
results on finiteness needed to motivate further constructions.

Following~\cite{BN:SIAMJMAA09}, a set of matrices $\setA\subset\Mset(N,M)$ is
called an \emph{independent row uncertainty set} (\emph{IRU-set}) if it
consists of all the matrices
\[
A=\left(\begin{array}{cccc}
a_{11}&a_{12}&\cdots&a_{1M}\\
a_{21}&a_{22}&\cdots&a_{2M}\\
\cdots&\cdots&\cdots&\cdots\\
a_{N1}&a_{N2}&\cdots&a_{NM}
\end{array}\right),
\]
wherein each of the rows $a_{i} = (a_{i1}, a_{i2}, \ldots, a_{iM})$ belongs
to some set of $M$-rows $\setA_{i}$, $i=1,2,\ldots,N$. Clearly, any singleton
set of matrices is an IRU-set. An IRU-set of matrices will be called
\emph{positive} if so are all its matrices which is equivalent to positivity
of all the rows constituting the sets $\setA_{i}$.

If the set $\setA$ is compact, which is equivalent to compactness of each set
of rows $\setA_{1},\setA_{2},\ldots,\setA_{N}$, then the following quantities
are well defined:
\[
\rho_{min}(\setA) = \min_{A \in \setA} \rho(A), \quad
\rho_{max}(\setA) = \max_{A \in \setA} \rho(A).
\]
We will use the notation
\[
\hat{\rho}_{n}(\setA)= \sup_{A_{i}\in\setA} \rho(A_{n} \cdots A_{1})^{1/n},\quad
\check{\rho}_{n}(\setA)= \inf_{A_{i}\in\setA} \rho(A_{n} \cdots A_{1})^{1/n}.
\]

As shows the following theorem, the finiteness conjecture is valid for
compact IRU-sets of positive matrices.

\begin{theorem}\label{T:BNP}
Let $\setA$ be a compact IRU-set of positive matrices and $\tilde{\setA}$ be
a compact set of matrices satisfying the inclusions
$\setA\subseteq\tilde{\setA}\subseteq\co(\setA)$, where $\co(\setA)$ stands
for the convex hull of the set $\setA$. Then
\begin{enumerate}[\rm(i)]
\item $\check{\rho}_{n}(\tilde{\setA})=\rho_{min}(\setA)$ for all $n\ge 1$,
    and therefore
    $\check{\rho}(\tilde{\setA})=\rho_{min}(\tilde{\setA})=\rho_{min}(\setA)$;

\item $\hat{\rho}_{n}(\tilde{\setA})=\rho_{max}(\setA)$ for all $n\ge 1$,
    and therefore
    $\hat{\rho}(\tilde{\setA})=\rho_{max}(\tilde{\setA})=\rho_{max}(\setA)$.
\end{enumerate}
\end{theorem}
For the cases $\tilde{\setA}=\setA$ and $\tilde{\setA}=\co(\setA)$ this
theorem in a somewhat different formulation is proved
in~\cite{NesPro:SIAMJMAA13}. The next example demonstrates that none of the
equalities $\rho_{max}(\tilde{\setA})=\rho_{max}(\setA)$ and
$\rho_{min}(\tilde{\setA})=\rho_{min}(\setA)$ holds for arbitrary sets of
matrices.
\begin{example}\rm
Consider the sets of matrices $\setA=\{A_{1},A_{2}\}$ and
$\setB=\{B_{1},B_{2}\}$, where
\[
A_{1}=\left(\begin{array}{cc}
0&2\\0&0
\end{array}\right),\quad
A_{2}=\left(\begin{array}{cc}
0&0\\2&0
\end{array}\right),\quad
B_{1}=\left(\begin{array}{cc}
2&0\\0&0
\end{array}\right),\quad
B_{2}=\left(\begin{array}{cc}
0&0\\0&2
\end{array}\right).
\]
Then $\rho_{max}(\setA)<\rho_{max}(\co(\setA))$ and
$\rho_{min}(\setB)>\rho_{min}(\co(\setB))$ because
\begin{align*}
\rho_{max}(\setA)=\max_{A\in\setA}\rho(A)&=0,\quad \rho_{max}(\co(\setA))=\max_{A\in\co(\setA)}\rho(A)\ge
\rho\left(\tfrac{1}{2}(A_{1}+A_{2})\right)=1,\\
\rho_{min}(\setB)=\min_{B\in\setB}\rho(B)&=2,\quad \rho_{min}(\co(\setB))=\min_{B\in\co(\setB)}\rho(B)\le
\rho\left(\tfrac{1}{2}(B_{1}+B_{2})\right)=1.
\end{align*}
\end{example}

\begin{remark*}\rm
If an IRU-set $\setA$ is formed by a set of rows
$\setA_{1},\setA_{2},\ldots,\setA_{N}$, then its convex hull $\co(\setA)$ is
the IRU-set formed by the set of rows $\co(\setA_{1}), \co(\setA_{2}),
\ldots, \co(\setA_{N})$.
\end{remark*}

\subsection{Hourglass alternative}
To prove Theorem~\ref{T:BNP} we will need some definitions and a number of
supporting facts. For vectors $x,y\in\bbR^{N}$, we write $x \ge y$ or $x>y$,
if all coordinates of the vector $x$ are not less or strictly greater,
respectively, than the corresponding coordinates of the vector $y$. Similar
notation will be applied to matrices.

In the space $\bbR^{1}$ of real numbers any two elements $x$ and $y$ are
\emph{comparable}, i.e. either $x\le y$ or $y\le x$. In this case we say that
the space $\bbR^{1}$ is \emph{linearly ordered}. In the spaces $\bbR^{N}$
with $N>1$ the situation is quite different. Here there exist infinitely many
pairs of non-comparable elements, and the failure of the inequality $x\ge y$
does not imply the inverse inequality $x\le y$. The existence of
noncomparable elements leads to the fact that if, for some $x$, the system of
linear inequalities
\[
Ax\ge v,\qquad A\in\setA\subset\Mset(N,M)
\]
has no solution, then it does not mean that for some matrix $\bar{A}\in\setA$
the inverse inequality $\bar{A}x\le v$ will be valid. Examples of
corresponding sets of matrices $\setA$ can be easily constructed. However, as
the following lemma shows, for the sets of matrices with independent row
uncertainty all is not so bad, and for linear inequalities an analogue of the
linear ordering of solutions holds.

\begin{lemma}\label{L:alternative}
Let $\setA\subset\Mset(N,M)$ be an IRU-set of matrices and let $\tilde{A}u=v$
for some matrix $\tilde{A}\in\setA$ and vectors $u,v$. Then the following
properties hold:
\begin{enumerate}[\rm H1:]
\item either $Au\ge v$ for all $A\in\setA$ or there exists a matrix
    $\bar{A}\in\setA$ such that $\bar{A}u\le v$ and $\bar{A}u\neq v$;

\item either $Au\le v$ for all $A\in\setA$ or there exists a matrix
    $\bar{A}\in\setA$ such that  $\bar{A}u\ge v$ and $\bar{A}u\neq v$.
\end{enumerate}
\end{lemma}

Assertions H1 and H2 have a simple geometrical interpretation. Imagine that
the sets $B_{l}=\{x:x\le v\}$ and $B_{u}=\{x:x\ge v\}$ form the lower and
upper bulbs of an hourglass with the neck at the point $v$. Then
Lemma~\ref{L:alternative} asserts that either all the grains $Au$ fill one of
the bulbs (upper or lower), or  there remains at least one grain in the other
bulb (lower or upper, respectively). Such an interpretation gives reason to
call Lemma~\ref{L:alternative} the \emph{hourglass alternative}. This
alternative will play a key role in the proof of Theorem~\ref{T:BNP} as well
as in its extension to a new class of matrices. The hourglass alternative has
been proposed by the author in~\cite{ACDDHK15} to analyze the minimax
relations between the spectral radii of matrix products.

\begin{proof}[Proof of Lemma~\ref{L:alternative}]
Given an IRU-set of matrices $\setA\subset\Mset(N,M)$, let $\tilde{A}u=v$ for
some matrix $\tilde{A}\in\setA$ and vectors $u,v>0$.

We first prove assertion H1. If the inequality $Au\ge v$ holds for all
$A\in\setA$ then there is nothing to prove. So let us suppose that for some
matrix $A=(a_{ij})\in\setA$ the inequality $Au\ge v$ is not satisfied. Then
representing the vectors $u$ and $v$ in the coordinate form
\[
u=(u_{1},u_{2},\ldots,u_{M})^{\transpose},\quad
v=(v_{1},v_{2},\ldots,v_{N})^{\transpose},
\]
we obtain that
\[
a_{i1}u_{1}+a_{i2}u_{2}+\cdots+a_{iM}u_{M}<v_{i}
\]
for some $i\in\{1,2,\ldots,N\}$; without loss of generality we can assume
that $i=1$. In this case, for the matrix
\[
\bar{A}=\left(\begin{array}{cccccc}
a_{11}&a_{12}&\cdots&a_{1M}\\
\tilde{a}_{21}&\tilde{a}_{22}&\cdots&\tilde{a}_{2M}\\
\cdots&\cdots&\cdots&\cdots\\
\tilde{a}_{N1}&\tilde{a}_{N2}&\cdots&\tilde{a}_{NM}
\end{array}\right),
\]
which is obtained from the matrix $\tilde{A}=(\tilde{a}_{ij})$ by changing
the first row to the row
\[
a_{1}=(a_{11},a_{12},\ldots,a_{1M})
\]
and therefore also belongs to $\setA$, we have the inequalities
\begin{align*}
a_{11}u_{1}+a_{12}u_{2}+\cdots+a_{1M}u_{M}&<v_{1}\\
\shortintertext{and}
\tilde{a}_{i1}u_{1}+\tilde{a}_{i2}u_{2}+\cdots+\tilde{a}_{iM}u_{M}&=v_{i},\qquad i=2,3,\ldots,N.
\end{align*}
Consequently, $\bar{A}u\le v$ and $\bar{A}u\neq v$, which completes the proof
of assertion~H1.

Assertion~H2 is proved similarly.
\end{proof}

We now show how Lemma~\ref{L:alternative} can be used to analyze the spectral
characteristics of sets of matrices. The \emph{spectral radius} of an
$(N\times N)$-matrix $A$ is defined as the maximal modulus of its eigenvalues
and denoted by $\rho(A)$. The spectral radius depends continuously on the
matrix, and in the case $A>0$ by the Perron-Frobenius
theorem~\cite[Theorem~8.2.2]{HJ2:e} the number $\rho(A)$ is a simple
eigenvalue of the matrix $A$ whereas all the other eigenvalues of $A$ are
strictly less than $\rho(A)$ by modulus. The eigenvector
$v=(v_{1},v_{2},\ldots,v_{N})^{\transpose}$ corresponding to the eigenvalue
$\rho(A)$ (normalized, for example, by the equality
$v_{1}+v_{2}+\cdots+v_{N}=1$) is uniquely determined and positive.

In the following lemma we consolidate some facts of the theory of positive
matrices, which in general are well known, but references to which are
spreaded among various publications.

\begin{lemma}\label{L:1}
Let $A$ be a non-negative $(N\times N)$-matrix, then the following assertions
hold:
\begin{enumerate}[\rm(i)]
\item if $Au\le\lambda u$ for some vector $u>0$, then $\lambda\ge0$ and
    $\rho(A)\le\lambda$;
\item moreover, if in conditions of {\rm(i)} $A>0$ and $Au\neq\lambda u$,
    then $\rho(A)<\lambda$;
\item if $Au\ge\lambda u$ for some non-zero vector $u\ge0$ and some number
    $\lambda\ge0$, then $\rho(A) \ge\lambda$;
\item moreover, if in conditions of {\rm(iii)} $A>0$ and $Au\neq\lambda u$,
    then $\rho(A)> \lambda$.
\end{enumerate}
\end{lemma}

\begin{proof}
As stated in~\cite[Corollary~8.1.29]{HJ2:e}, for any nonnegative matrix $A$
and numbers $\alpha,\beta\ge0$, the inequalities
\begin{equation}\label{E:horn29}
\alpha\le \rho(A)\le \beta
\end{equation}
are valid provided that $\alpha u\le Au\le \beta u$ for some $u>0$, from
which assertion (i) immediately follows. Let us prove three remaining
assertions.

(ii) Let  $Au\le\lambda u$ for $u>0$, where $A>0$ and $Au\neq\lambda u$. Then
at least one coordinate of the vector $Au-\lambda u\le 0$ is strictly
negative. Therefore, the condition $A>0$ implies strict negativity of all the
coordinates of the vector $A(Au-\lambda u)$. Then there exists
$\varepsilon>0$ such that $A(Au-\lambda u)\le -\varepsilon u$ and therefore
$A^{2}u=A(Au-\lambda u)+\lambda Au\le (\lambda^{2}-\varepsilon)u$. Then, by
\eqref{E:horn29}, we get $\rho(A^{2})\le\lambda^{2}-\varepsilon$, and thus
$\rho(A)\le\sqrt{\lambda^{2}-\varepsilon}<\lambda$.

(iii) The condition $Au\ge\lambda u$ with a non-zero $u\ge0$ implies
$A^{n}u\ge \lambda^{n}u$ for any $n\ge1$. Then
$\|A^{n}\|\cdot\|u\|\ge\|A^{n}u\|\ge\lambda^{n}\|u\|$, where $\|\cdot\|$ is
any norm monotone with respect to coordinates of a non-negative vector, e.g.
the Euclidean norm or the max-norm. Therefore, $\|A^{n}\|\ge \lambda^{n}$,
and by Gelfand's formula~\cite[Corollary~5.6.14]{HJ2:e}
$\rho(A)=\lim_{n\to\infty}\|A^n\|^{1/n}\ge\lambda$.

(iv) Now let $A>0$ and $Au\neq\lambda u$. Then at least one coordinate of the
vector $Au-\lambda u\ge 0$ is strictly positive. Therefore, the condition
$A>0$ implies strict positivity of all coordinates of the vector
$A(Au-\lambda u)$. Then there exists $\varepsilon>0$ such that $A(Au-\lambda
u)\ge \varepsilon u$ and therefore $A^{2}u=A(Au-\lambda u)+\lambda Au\ge
(\lambda^{2}+\varepsilon)u$. This, by assertion (iii) applied to the matrix
$A^{2}$, implies $\rho(A^{2})\ge\lambda^{2}+\varepsilon$, and thus
$\rho(A)\ge\sqrt{\lambda^{2}+\varepsilon}>\lambda$.

The lemma is proved.
\end{proof}

The proofs of Lemmas~\ref{L:alternative} and \ref{L:1} are borrowed
from~\cite{ACDDHK15} and presented here only for the sake of completeness of
presentation. Lemma~\ref{L:1} resembles Lemma~1
from~\cite{NesPro:SIAMJMAA13}. The next lemma shows that for the IRU-sets of
positive matrices there are valid assertions in a certain sense inverse to
Lemma~\ref{L:1}.
\begin{lemma}\label{L:mainmin}
Let $\setA\subset\Mset(N,N)$ be a compact IRU-set of positive matrices, then
the following assertions hold:
\begin{enumerate}[\rm(i)]
\item if $\tilde{A}\in\setA$ is a matrix satisfying $\rho(\tilde{A}) =
    \rho_{min}(\setA)$ and $\tilde{v}$ is its positive eigenvector
    corresponding to the eigenvalue $\rho(\tilde{A})$, then $A\tilde{v}\ge
    \rho_{min}(\setA)\tilde{v}$ for all $A\in\setA$;

\item if $\tilde{A}\in\setA$ is a matrix satisfying $\rho(\tilde{A}) =
    \rho_{max}(\setA)$ and $\tilde{v}$ is its positive eigenvector
    corresponding to the eigenvalue $\rho(\tilde{A})$, then $A\tilde{v}\le
    \rho_{max}(\setA)\tilde{v}$ for all $A\in\setA$.
\end{enumerate}
\end{lemma}

\begin{proof}
To prove assertion (i) let us note that
$\tilde{A}\tilde{v}=\rho_{min}(\setA)\tilde{v}$. Then by assertion (i) of
Lemma~\ref{L:alternative} either $A\tilde{v}\ge\rho_{min}(\setA)\tilde{v}$
for all $A\in\setA$ or there exists a matrix $\bar{A}\in\setA$ such that
$\bar{A}\tilde{v}\le \rho_{min}(\setA)\tilde{v}$ and $\bar{A}\tilde{v}\neq
\rho_{min}(\setA)\tilde{v}$. In the latter case, by Lemma~\ref{L:1} there
would be valid the inequality $\rho(\bar{A})<\rho_{min}(\setA)$ which
contradicts to the definition of $\rho_{min}(\setA)$. Hence, the inequality
$A\tilde{v}\ge\rho_{min}(\setA)\tilde{v}$ holds for all $A\in\setA$.

Assertion (ii) is proved similarly.
\end{proof}

Now all is ready to prove Theorem~\ref{T:BNP}.

\subsection{Proof of Theorem~\ref{T:BNP}}
To prove assertion (i) choose a matrix $\tilde{A}\in\setA$ for which
$\rho(\tilde{A})=\rho_{min}(\setA)$ and denote by
$\tilde{v}=(\tilde{v}_{1},\tilde{v}_{2},\ldots,\tilde{v}_{N})^\transpose$ a
positive eigenvector of $\tilde{A}$ corresponding to the eigenvalue
$\rho(\tilde{A})$. Then by assertion (i) of Lemma~\ref{L:mainmin}
$A\tilde{v}\ge\rho_{min}(\setA)\tilde{v}$ for all $A\in\co(\setA)$ and
therefore for all $A\in\tilde{\setA}$. Hence $A_{i_{n}}\cdots
A_{i_{1}}\tilde{v}\ge\rho_{min}^{n}(\setA)\tilde{v}$ for all
$A_{i_{j}}\in\tilde{\setA}$. Consequently, by Lemma~\ref{L:1}
$\rho(A_{i_{n}}\cdots A_{i_{1}})\ge\rho_{min}^{n}(\setA)$ for all
$A_{i_{j}}\in\tilde{\setA}$ and therefore
\[
\check{\rho}_{n}(\tilde{\setA})\ge\rho_{min}(\setA).
\]

On the other hand, since $\tilde{A}\in\setA\subseteq\tilde{\setA}$ then
clearly
\[
\check{\rho}_{n}(\tilde{\setA})\le\rho(\tilde{A}^{n})^{1/n}=\rho_{min}(\setA),
\]
which implies $\check{\rho}_{n}(\tilde{\setA})=\rho_{min}(\setA)$ for all
$n\ge 1$. Observing now that
$\check{\rho}_{1}(\tilde{\setA})=\rho_{min}(\tilde{\setA})$ we complete the
proof of assertion (i).

The proof of assertion (ii) is carried out by verbatim repetition of the
proof of assertion (i) by taking instead of $\tilde{A}$ a matrix maximizing
the spectral radii of matrices from $\setA$ and instead of the estimates for
$\rho_{min}(\setA)$ the corresponding estimates for $\rho_{max}(\setA)$, and
then using assertion (ii) of Lemma~\ref{L:mainmin} instead of assertion (i).

\section{$\Hset$-sets of matrices}\label{S:Hsets}
Apart from general properties of positive matrices given in Lemma~\ref{L:1},
the proof of Theorem~\ref{T:BNP} relies only on those properties of IRU-sets
of matrices which were formulated in Lemma~\ref{L:alternative} as statements
H1 and H2 of the hourglass alternative. Therefore, it is natural to
axiomatize the class of matrices satisfying the statements H1 and H2, and to
study its properties.

\subsection{Main definitions}

A set of positive matrices $\setA\subset\Mset(N,M)$ will be called
\emph{$\Hset$-set} or \emph{hourglass set} if every time the equality
$\tilde{A}u=v$ is true for some matrix $\tilde{A}\in\setA$ and vectors
$u,v>0$ there are also true assertions H1 and H2 of
Lemma~\ref{L:alternative}.

A trivial example of $\Hset$-sets are \emph{linearly ordered} sets
$\setA=\{A_{1}$, $A_{2}$, \ldots, $A_{n}\}$ composed of positive matrices
$A_{i}$ satisfying the inequalities $0<A_{1}<A_{2}<\cdots<A_{n}$. In this
case, for each $u>0$, the vectors $A_{1}u,A_{2}u,\ldots,A_{n}u$ are strictly
positive and linearly ordered, which yields the validity of assertions H1 and
H2 for $\setA$. A less trivial and more interesting example of $\Hset$-sets,
as follows from Lemma~\ref{L:alternative}, is the class of sets of positive
matrices with independent row uncertainty.

Not every set of positive matrices is an $\Hset$-set. A relevant example
could easily be built for the set $\setA=\{A,B\}$ consisting of two $(2\times
2)$-matrices. In this case, for $\setA$ was $\Hset$-set, it is necessary that
the vectors $Au$ and $Bu$ were comparable for any vector $u>0$, that is,
either $Au\le Bu$ or $Bu\le Au$. But this is not fulfilled, for example, in
the case when $AB=P$, where $P$ is any projection on the linear hull of the
vector $(-1,1)$.

Let us describe some general properties of the class of $\Hset$-sets of
matrices. Introduce the operations of Minkowski addition and multiplication
for sets of matrices:
\[
\setA+\setB=\{A+B:A\in\setA,~ B\in\setB\},\quad
\setA\setB=\{AB:A\in\setA,~ B\in\setB\},
\]
and also the operation of multiplication of a set of matrices by a scalar:
\[
t\setA=\setA t=\{tA:t\in\bbR,~A\in\setA\}.
\]
Naturally, the operation of addition is \emph{admissible} if and only if the
matrices from the sets $\setA$ and $\setB$ are of the same size, while the
operation of multiplication is \emph{admissible} if and only if the sizes of
the matrices from sets $\setA$ and $\setB$ are matched: dimension of the rows
of the matrices from $\setA$ is the same as dimension of the columns of the
matrices from $\setB$. Problems with matching of sizes do not arise when one
considers sets of square matrices of the same size.

In what follows, we will need to make various kinds of limiting transitions
with the matrices from the sets under consideration as well as with the sets
of matrices themselves. In this connection, it is natural to restrict our
considerations to only compact (closed and bounded) sets of matrices. The
totality of all compact sets of positive $(N\times M)$-matrices with
independent row uncertainty will be denoted by $\IRU(N,M)$. The totality of
all finite\footnote{We will not consider infinite sets since this article is
not a proper place to get into the intricacies of determining the linear
ordering for infinite sets.} positive linearly ordered sets of $(N\times
M)$-matrices will be denoted by $\Lin(N,M)$. At last, by $\Hset(N,M)$ we
denote the set of all compact $\Hset$-sets of positive $(N\times
M)$-matrices.

\begin{theorem}\label{T:semiring} The following is true:
\begin{enumerate}[\rm(i)]
  \item $\setA+\setB\in\Hset(N,M)$ if $\setA,\setB\in\Hset(N,M)$;
  \item $\setA\setB\in\Hset(N,Q)$ if $\setA\in\Hset(N,M)$ and
      $\setB\in\Hset(M,Q)$;
  \item $t\setA=\setA t\in\Hset(N,M)$ if $t>0$ and $\setA\in\Hset(N,M)$.
\end{enumerate}
\end{theorem}

\begin{proof}
First prove (i). Show the validity of assertion H1 for the sum $\setA+\setB$.
Let, for some matrix $C\in\setA+\setB$ and vectors $u,v>0$, the equality
$Cu=v$ holds. Then, by definition of the set $\setA+\setB$, there exist
matrices $\tilde{A}\in\setA$ and $\tilde{B}\in\setB$ such that
$C=\tilde{A}+\tilde{B}$, and hence $(\tilde{A}+\tilde{B})u=v$. Denote
$\tilde{A}u=v_{1}$ and $\tilde{B}u=v_{2}$ then $v_{1}+v_{2}=v$, where
$v_{1},v_{2}>0$ due to the positivity of the matrices $\tilde{A}$ and
$\tilde{B}$. If
\begin{equation}\label{E:ABgeforall}
Au\ge v_{1},~ Bu\ge v_{2}\quad\text{for all}\quad  A\in\setA,~ B\in\setB,
\end{equation}
then, for all $A+B\in\setA+\setB$, there will be valid also the inequality
$(A+B)u\ge v_{1}+v_{2}=v$. Thus, in this case assertion H1 is proven.

Now, let \eqref{E:ABgeforall} fail, and let, to be specific, the inequality
$Au\ge v_{1}$ be not valid for at least one matrix $A\in\setA$. Then, since
$\setA\in\Hset(N,M)$, assertion H1 for the set of matrices $\setA$ implies
the existence of a matrix $\bar{A}\in\setA$ such that $\bar{A}u\le v_{1}$ and
$\bar{A}u\neq v_{1}$. In this case the matrix
$\bar{A}+\tilde{B}\in\setA+\setB$ will satisfy the inequality
$(\bar{A}+\tilde{B})u\le v_{1}+v_{2}=v$, and moreover,
$(\bar{A}+\tilde{B})u\neq v$ since $\tilde{B}u= v_{2}$ while $\bar{A}u\neq
v_{1}$. Thus, assertion H1 for the set $\setA+\setB$ is also valid in the
case when \eqref{E:ABgeforall} fails.

The proof of assertion H2 for the set $\setA+\setB$ is similar. Compactness
of the set $\setA+\setB$ in the case when the sets $\setA$ and $\setB$ are
compact is evident.

We now prove (ii). Show the validity of assertion H1 for the product
$\setA\setB$. Suppose that $Cu=v$ for some matrix $C\in\setA\setB$ and
vectors $u,v>0$. Then, by definition of the set $\setA\setB$, there exist
matrices $\tilde{A}\in\setA$ and $\tilde{B}\in\setB$ such that
$\tilde{A}\tilde{B}u=v$. By denoting $\tilde{B}u=w$ we obtain, due to the
positivity of the matrix $\tilde{B}$ and the vector $u$, that $w>0$ and
$\tilde {A}w=u$. If
\begin{equation}\label{E:ABgeforall1}
Aw\ge v,~ Bu\ge w\quad\text{for all}\quad  A\in\setA,~ B\in\setB,
\end{equation}
then, due to the positivity of the matrices $A\in\setA$ and $B\in\setB$, for
all $AB\in\setA\setB$ there will be valid also the inequalities $ABu\ge Aw\ge
v$. Thus in this case assertion H1 is proved.

Now, let \eqref{E:ABgeforall1} fail, and let, to be specific, the inequality
$Bu\ge w$ be not valid for at least one matrix $B\in\setB$. Then, since
$\setB\in\Hset(N,M)$, assertion H1 for the set of matrices $\setB$ implies
the existence of a matrix $\bar{B}\in\setB$ such that $\bar{B}u\le w$ and
$\bar{B}u\neq w$. But in this case the matrix
$\tilde{A}\bar{B}\in\setA\setB$, due to the positivity of the matrix
$\tilde{A}$, will satisfy the inequality $\tilde{A}\bar{B}u\le \tilde{A}w=v$,
and then $\tilde{A}\bar{B}u\le v$. Moreover, $\tilde{A}\bar{B}u\neq v$ since
$\bar{B}u\le w$, $\bar{B}u\neq w$ and the matrix $\tilde{A}$ is positive.
Thus, assertion H1 for the set $\setA\setB$ is also valid in the case when
\eqref{E:ABgeforall1} fails.

The proof of assertion H2 for the set $\setA\setB$ is similar. Compactness of
the set $\setA\setB$ in the case when the sets $\setA$ and $\setB$ are
compact is evident.

The proof of (iii) is also evident.
\end{proof}

\begin{remark*}\rm
The requirement of positivity for the matrices and the vectors $u,v$ in the
definition of $\Hset$-sets was introduced to ensure the inclusion
$\setA\setB\in\Hset(N,Q)$ in Theorem~\ref{T:semiring}, as well as to provide
an opportunity to further use of Lemma~\ref{L:1} for the analysis of the
spectral properties of the sets of matrices from $\Hset(N,Q)$.
\end{remark*}

By Theorem~\ref{T:semiring} the totality of sets of square matrices
$\Hset(N,N)$ is enabled by additive and multiplicative group operations, but
itself is not a group, neither additive nor multiplicative. However, after
adding the zero additive element $\{0\}$ and the multiplicative identity
element $\{I\}$ to $\Hset(N,N)$, the resulting totality
$\Hset(N,N)\cup\{0\}\cup\{I\}$ becomes a semiring~\cite{Golan99}.

Theorem~\ref{T:semiring} implies that any finite sum of any finite products
of sets of matrices from $\Hset(N,N)$ is again a set of matrices from
$\Hset(N,N)$. Moreover, for any integers $n,d\ge1$, all the polynomial sets
of matrices
\begin{equation}\label{E:poly}
P(\setA_{1},\setA_{1},\ldots,\setA_{n})=
\sum_{k=1}^{d}\sum_{i_{1},i_{2},\ldots,i_{k}\in\{1,2,\ldots,n\}}
p_{i_{1},i_{2},\ldots,i_{k}}\setA_{i_{1}}\setA_{i_{2}}\cdots\setA_{i_{k}},
\end{equation}
where $\setA_{1},\setA_{1},\ldots,\setA_{n}\in\Hset(N,N)$ and the scalar
coefficients $p_{i_{1},i_{2},\ldots,i_{k}}$ are positive, belong to the set
$\Hset(N,N)$.

The polynomials \eqref{E:poly} allow to construct not only the elements
$P(\setA_{1},\setA_{1},\ldots,\setA_{n})$ of the set $\Hset(N,N)$ but also
the elements of arbitrary sets $\Hset(N,M)$, by taking the arguments
$\setA_{1},\setA_{1},\ldots,\setA_{n}$ from the sets $\Hset(N_{i},M_{i})$
with arbitrary matrix sizes $N_{i}\times M_{i}$. One must only ensure that
the products $\setA_{i_{1}}\setA_{i_{2}}\cdots\setA_{i_{k}}$ would be
admissible and determine the sets of matrices of dimension $N\times M$.

We have presented above two types of non-trivial $\Hset$-sets of matrices,
the sets of matrices with independent row uncertainty and the linearly
ordered sets of positive matrices. In this connection, let us denote by
$\Hset_{*}(N,M)$ the totality of all sets of $(N\times M)$-matrices which can
be obtained as admissible finite sums of finite products of the sets of
positive matrices with independent rows uncertainty or the sets of linearly
ordered positive matrices. In other words, $\Hset_{*}(N,M)$ is the totality
of all sets of matrices that can be represented as the values of polynomials
\eqref{E:poly} with the arguments taken from the sets of the matrices
belonging to $\IRU(N_{i},M_{i})\cup\Lin(N_{i},M_{i})$.

\begin{question*}
Does equality $\Hset_{*}(N,M)=\Hset(N,M)$ hold?
\end{question*}

The answer to this question is probably negative, but we do not aware of any
counterexamples.

\subsection{Closure of the set $\Hset(N,M)$}

When considering various types of problems related to the sets of matrices,
it is desirable to be able to perform limit transitions. In fact, for further
goals we would like to be able to extend some facts relevant to $\Hset$-sets
of positive matrices to the same kind of sets of matrices, but with
non-negative elements. To achieve this, without going too deep into the
variety of all topologies on spaces of subsets, we confine ourselves to the
description of only one of them, the topology specified by the Hausdorff
metric.

Given some matrix norm $\|\cdot\|$ on $\Mset(N,M)$, denote by $\Kset(N,M)$
the totality of all compact subsets of $\Mset(N,M)$. Then for any two sets of
matrices $\setA,\setB\in\Kset(N,M)$ there is defined the \emph{Hausdorff
metric}
\[
H(\setA,\setB)=
\max\left\{\adjustlimits\sup_{A\in\setA}\inf_{B\in\setB}\|A-B\|,
~\adjustlimits\sup_{B\in\setB}\inf_{A\in\setA}\|A-B\|\right\},
\]
in which $\Kset(N,M)$ becomes a full metric space. Then $\Hset(N,M)\subset
\Kset(N,M)$, equipped with the Hausdorff metric, also becomes a metric space.

As is known, see, e.g.,~\cite[Chapter~E, Proposition~5]{Ok07}, any mapping
$F(\setA)$ acting from $\Kset(N,M)$ into itself is continuous in the
Hausdorff metric at some point $\setA_{0}$ if and only if it is both upper
and lower semicontinuous. It is also known~\cite[Section~1.3]{BGMO:84:e} that
the mappings
\[
(\setA,\setB)\mapsto \setA+\setB,\quad (\setA,\setB)\mapsto \setA\setB,\quad
\setA\mapsto\setA\times\setA\times\cdots\times\setA,\quad
\setA\mapsto\co(\setA),
\]
where $\setA$ and $\setB$ are compact sets, are both upper and lower
semicontinuous. Then these mappings are continuous in the Hausdorff metric,
and any polynomial mapping \eqref{E:poly} possesses the same continuity
properties.

Denote by $\overline{\Hset}(N,M)$ the closure of the set $\Hset(N,M)$ in the
Hausdorff metric. It is obvious that $\{0\},\{I\}\in\overline{\Hset}(N,M)$,
and since the Minkowski addition and multiplication of matrix sets are
continuous in the Hausdorff metric, then by Theorem~\ref{T:Hset} the set
$\overline{\Hset}(N,N)$ is a semiring. However, the answer to the question
when, for some $\setA$, the inclusion $\setA\in\overline{\Hset}(N,M)$ holds,
requires further analysis. We restrict ourselves to the description of only
one case where the answer to this question can be given explicitly.

Let $\1$ stand for the matrix (of appropriate size) with all elements equal
to $1$. First note that each set of \emph{linearly ordered non-negative
matrices} $\setA=\{A_{1},A_{2},\ldots,A_{n}\}$, that is a set whose matrices
$A_{i}$ satisfy the inequalities $0\le A_{1}\le A_{2}\le \cdots\le A_{n}$, is
a limiting point in the Hausdorff metric of the family of linearly ordered
sets of positive matrices
\[
\setA(\varepsilon)=
\{A_{1}+\varepsilon\1,A_{2}+2\varepsilon\1,\ldots,A_{n}+n\varepsilon\1\},\qquad \varepsilon>0.
\]

Further, let $\setA$ be an IRU-set of non-negative matrices. Then any set of
matrices
\[
\setA(\varepsilon)=\setA+\varepsilon\1=\left\{A+\varepsilon\1:A\in\setA\right\},\qquad \varepsilon>0,
\]
is also an IRU-set, but this time consisting of positive matrices. To verify
this, it suffices to note that if a set of matrices $\setA$ is defined by
sets of rows $\setA_{i}$ then the set $\setA+\varepsilon\1$ will be defined
by the sets of rows
$\setA_{i}+\varepsilon\1=\{a+\varepsilon\1:a\in\setA_{i}\}$, where $\1$ is
the unit row of appropriate size. Moreover, this IRU-set $\setA$ of
non-negative matrices, as well as in the previous case, will be a limiting
point in the Hausdorff metric for the positive family of IRU-sets
$\setA(\varepsilon)$, $\varepsilon>0$.

These observations imply the following lemma.
\begin{lemma}\label{L:polynonneg}
The values of any polynomial mapping \eqref{E:poly} with the arguments from
finite linearly ordered sets of non-negative matrices or from IRU-sets of
non-negative matrices belong to the closure in the Hausdorff metric of the
totality of positive $\Hset$-sets of matrices.
\end{lemma}

\section{Main results}\label{S:main}

In this section we present a generalization of Theorem~\ref{T:BNP} to the
case of $\Hset$-sets of matrices.

\begin{theorem}\label{T:Hset}
Let $\setA\in\overline{\Hset}(N,N)$, and let $\tilde{\setA}$ be a set of
matrices satisfying the inclusions
$\setA\subseteq\tilde{\setA}\subseteq\co(\setA)$. Then
\begin{enumerate}[\rm(i)]
\item $\check{\rho}_{n}(\tilde{\setA})=\rho_{min}(\setA)$ for all $n\ge 1$,
    and therefore
    $\check{\rho}(\tilde{\setA})=\rho_{min}(\tilde{\setA})=\rho_{min}(\setA)$;

\item $\hat{\rho}_{n}(\tilde{\setA})=\rho_{max}(\setA)$ for all $n\ge 1$,
    and therefore
    $\hat{\rho}(\tilde{\setA})=\rho_{max}(\tilde{\setA})=\rho_{max}(\setA)$.
\end{enumerate}
\end{theorem}

\begin{proof}
If $\setA\in\Hset(N,N)$ then, by definition, the set $\setA$ consists of
positive matrices. Therefore, for  $\setA\in\Hset(N,N)$ assertions H1 and H2
of Lemma~\ref{L:alternative} hold, which implies that Lemma~\ref{L:mainmin}
is valid, too. Then the proof of the theorem word for word repeats the proof
of Theorem~\ref{T:BNP}. Thus, we need only consider the case when
$\setA\in\overline{\Hset}(N,N)$ but $\setA\not\in\Hset(N,N)$.

First prove that, for every $n\ge1$, there are valid the equalities
\begin{equation}\label{E:needed}
\check{\rho}_{n}(\setA)=\rho_{min}(\setA),\quad
\check{\rho}_{n}(\co(\setA))=\rho_{min}(\setA)
\end{equation}

Since $\setA\in\overline{\Hset}(N,N)$ then there exists a sequence of sets of
matrices $\setA_{k}\in\Hset(N,N)$, $k=1,2,\dotsc$, converging to $\setA$ in
the Hausdorff metric. Then, as it has been already proved, for each $n,k\ge
1$ we have the equalities
\begin{equation}\label{E:needed-k}
\check{\rho}_{n}(\setA_{k})=\rho_{min}(\setA_{k}),\quad
\check{\rho}_{n}(\co(\setA_{k}))=\rho_{min}(\setA_{k}),
\end{equation}
and therefore it is natural to try to get \eqref{E:needed} by limiting
transition from \eqref{E:needed-k}. To do this, we recall the following
simplified version of Berge's Maximum Theorem~\cite[Chapter~E,
Section~3]{Ok07}.
\begin{lemma}\label{L:Berge}
Let $X$ and $Y$ be metric spaces, $\Gamma: X\to Y$ be a multivalued mapping
with compact values, and $\varphi$ be a continuous real function on $X\times
Y$. If the mapping $\Gamma$ is continuous, that is both upper and lower
semicontinuous, at a point $x_{0}\in X$ then both functions
$M(x)=\max_{y\in\Gamma(x)}\varphi(x,y)$ and
$m(x)=\min_{y\in\Gamma(x)}\varphi(x,y)$ are also continuous at the point
$x_{0}$.
\end{lemma}

To use this lemma we will treat $\Mset(N,N)$ as a metric space, and take the
following notation:
\begin{alignat*}{2}
X&=\Kset(N,N),&\qquad Y&=\underbrace{\Mset(N,N)\times\cdots\times\Mset(N,N)}_{n~\text{times}},\\
x&=\setA\in X,&\qquad
y&=(A_{1},\ldots,A_{n})\in Y,\\
\Gamma(x)&=\setA\times\cdots\times\setA,&\qquad
\varphi(x,y)&\equiv\varphi(y)=\rho(A_{n}\cdots A_{1}),
\end{alignat*}
Here, the function $\varphi(x,y)$, which in fact depends on a single argument
$y$, is continuous. The multivalued mapping $\Gamma(x)$, for each
$x=\setA\in\Kset(N,N)$, takes compact values and is also continuous in the
Hausdorff metric, see, e.g.,~\cite[Section~1.3]{BGMO:84:e}. Therefore,
$\min_{y\in\Gamma(x)}\varphi(x,y)=\rho_{min}(\setA)$, and by
Lemma~\ref{L:Berge} the function $\rho_{min}(\setA)$ is continuous in
$\setA\in\Kset(N,N)$. Similarly, choosing as $\varphi(x,y)$ the functions of
the form $\varphi(x,y)\equiv\varphi(y)=\rho(A_{n}\cdots A_{1})^{1/n}$ for
various $n\ge 1$, we obtain from Lemma~\ref{L:Berge} continuity of the
functions $\check{\rho}_{n}(\setA)$ in $\setA\in\Kset(N,N)$ for all $n\ge 1$.
And choosing as $\Gamma(x)$ the multivalued mapping
$\Gamma(x)=\co(\setA)\times\cdots\times\co(\setA)$, which also takes compact
values and is continuous in the Hausdorff metric because in the Hausdorff
metric it is continuous the mapping
$\setA\mapsto\co(\setA)$~\cite[Section~1.3]{BGMO:84:e}, we obtain similarly
that the functions $\check{\rho}_{n}(\co(\setA))$ are continuous in
$\setA\in\Kset(N,N)$ for every $n\ge 1$.

Thus, we have shown that all the functions in equalities \eqref{E:needed-k}
are continuous in $\setA\in\Kset(N,N)$ from which, taking the limit as
$\setA_{k}\to\setA\in\overline{\Hset}(N,N)$, we obtain \eqref{E:needed}.

Let now $\tilde{\setA}$ be a compact set of matrices satisfying the
inclusions $\setA\subseteq\tilde{\setA}\subseteq\co(\setA)$. Then, since the
quantities $\check{\rho}_{n}(\cdot)$ are defined as infima over the
corresponding sets, we have the inequalities
\[
\check{\rho}_{n}(\co(\setA))\le\check{\rho}_{n}(\tilde{\setA})\le
\check{\rho}_{n}(\setA).
\]
Therefore, by virtue of the already proven equalities \eqref{E:needed}, for
each $n\ge 1$, there holds the equality
\[
\check{\rho}_{n}(\tilde{\setA})=\rho_{min}(\setA),
\]
and then, due to \eqref{E-LSRad},
$\check{\rho}(\tilde{\setA})=\rho_{min}(\setA)$. Finally, observe that, by
definition, $\check{\rho}_{1}(\tilde{\setA})=\rho_{min}(\tilde{\setA})$ and
then $\rho_{min}(\tilde{\setA})=\rho_{min}(\setA)$. Assertion (i) of
Theorem~\ref{T:Hset} is completely proved.

The proof of assertion (ii) is similar.
\end{proof}

On application of Theorem~\ref{T:Hset}, among the first there arises the
question about verification of the inclusion $\setA\in\overline{\Hset}(N,N)$,
for given sets of matrices $\setA$. One such case has been described in
Lemma~\ref{L:polynonneg}, which implies the following corollary.
\begin{corollary}\label{C1}
Let $\setA$ be a set of matrices obtained as the value of a polynomial
mapping \eqref{E:poly}, whose arguments are finite linearly ordered sets of
non-negative matrices or compact IRU-sets of non-negative matrices. Then for
any compact set of matrices $\tilde{\setA}$ satisfying the inclusions
$\setA\subseteq\tilde{\setA}\subseteq\co(\setA)$ assertions of
Theorem~\ref{T:Hset} hold.
\end{corollary}

\section{Spectral characteristics of convex hulls of matrix sets}\label{S:cosets}

Theorem~\ref{T:Hset} implies that
\begin{equation}\label{E:JSR-LSR-conv}
\hat{\rho}(\setA)=\hat{\rho}(\co(\setA)),\quad
\check{\rho}(\setA)=\check{\rho}(\co(\setA))
\end{equation}
for any set $\setA\in\overline{\Hset}(N,M)$. In fact, it is
known~\cite{Theys:PhD05,Jungers:09} that the first of
equalities~\eqref{E:JSR-LSR-conv} holds for arbitrary (not necessarily
non-negative) sets of matrices $\setA\subset\Mset(N,N)$, which follows from
the obvious observation that
\[
\sup_{A_{i}\in\setA}\|A_{n}\cdots
A_{1}\|=\sup_{A_{i}\in\co(\setA)}\|A_{n}\cdots A_{1}\|
\]
for any norm. The second equality in \eqref{E:JSR-LSR-conv} for general sets
of matrices is not true, as is seen from the example of the set
$\setA=\{I,-I\}$, for which $\check{\rho}(\setA)=1$ while
$\check{\rho}(\co(\setA))=0$. In this regard, we note the following general
assertion.

\begin{theorem}\label{T:conv}
For any bounded set of non-negative matrices $\setA\subset\Mset(N,N)$ the
second of equalities \eqref{E:JSR-LSR-conv} holds.
\end{theorem}

\begin{proof}
We will need some auxiliary facts. Let us take in the
definition~\eqref{E-LSRad0} the norm $\|x\|=\sum_{i}|x_{i}|$ and notice that
in this case $\|x\|=\sum_{i}x_{i}$ for any
$x=(x_{1},x_{2},\ldots,x_{N})^{\transpose}\ge 0$, which implies that
\begin{equation}\label{E:sumin1}
\left\|\sum u_{i}\right\|=\sum_{i}\|u_{i}\|
\end{equation}
for any finite set of vectors $u_{i}\ge0$. Notice also that
\begin{equation}\label{E:srbound}
\|Ae\|\ge\rho(A),\quad\text{where}\quad e=(1,1,\ldots,1)^{\transpose},
\end{equation}
for any  matrix $A\ge0$. Indeed, if inequality~\eqref{E:srbound} is not true
then $\|Ae\|<\rho(A)$, which means that all coordinates of $Ae$ are less than
$\rho(A)$, i.e. $Ae<\rho(A)e$. This leads, by assertion (i) of
Lemma~\ref{L:1}, to the self-contradictory inequality $\rho(A)<\rho(A)$.

To prove the equality $\check{\rho}(\co(\setA))=\check{\rho}(\setA)$ let us
observe first that
\begin{equation}\label{E:checkrho}
\check{\rho}(\co(\setA))\le \check{\rho}(\setA),
\end{equation}
since $\setA\subseteq\co(\setA)$ while due to the definition~\eqref{E-LSRad}
both sides of this inequality are infima of the same expression over
$\co(\setA)$ and $\setA$ respectively.

Now, given $n\ge 1$, let for each $i=1,2,\ldots,n$ a matrix $A_{i}$ be a
finite convex combinations of matrices $A^{(i)}_{j}\in\setA_{i}$, that is
$A_{i}=\sum_{j} \mu^{(i)}_{j}A^{(i)}_{j}\in\co(\setA)$  with some
$\mu^{(i)}_{j}\ge0$ and $\sum_{j} \mu^{(i)}_{j}=1$. Then in view of
\eqref{E:sumin1}%
\begin{multline*}
\|A_{n}\cdots A_{1}\|\cdot\|e\|\ge\|A_{n}\cdots A_{1}e\|=
\biggl\|\biggl(\sum_{j_{n}} \mu^{(n)}_{j_{n}}A^{(n)}_{j_{n}}\biggr)\cdots
\biggl(\sum_{j_{1}} \mu^{(1)}_{j_{1}}A^{(1)}_{j_{1}}\biggr)e\biggr\|=\\=
\biggl\|\sum_{j_{n}}\cdots\sum_{j_{1}}\mu^{(n)}_{j_{n}}\cdots \mu^{(1)}_{j_{1}}
A^{(n)}_{j_{n}}\cdots A^{(1)}_{j_{1}}e\biggr\|=
\sum_{j_{n}}\cdots\sum_{j_{1}}\mu^{(n)}_{j_{n}}\cdots \mu^{(1)}_{j_{1}}
\|A^{(n)}_{j_{n}}\cdots A^{(1)}_{j_{1}}e\|.
\end{multline*}
Here $\|e\|=N$, and by~\eqref{E:srbound} $\|A^{(n)}_{j_{n}}\cdots
A^{(1)}_{j_{1}}e\|\ge \rho(A^{(n)}_{j_{n}}\cdots
A^{(1)}_{j_{1}})\ge\check{\rho}_{n}(\setA)$. Therefore,
\[
N\|A_{n}\cdots A_{1}\|\ge
\biggl(\sum_{j_{n}}\cdots\sum_{j_{1}}\mu^{(n)}_{j_{n}}\cdots \mu^{(1)}_{j_{1}}\biggr)
\check{\rho}_{n}(\setA).
\]
Moreover, since $\sum_{j_{n}}\ldots\sum_{j_{1}}\mu^{(n)}_{j_{n}}\cdots
\mu^{(1)}_{j_{1}}=1$ then
\[
\|A_{n}\cdots A_{1}\|\ge \tfrac{1}{N}\,
\check{\rho}_{n}(\setA).
\]
Taking in this last inequality the infimum over all $A_{1},\ldots,
A_{n}\in\co(\setA)$ we obtain the inequalities
\[
\inf_{A_{i}\in\co(\setA)}\|A_{n}\cdots A_{1}\|\ge \tfrac{1}{N}\,
\check{\rho}_{n}(\setA),\qquad n\ge 1,
\]
which we substitute in \eqref{E-LSRad0}:
\[
\check{\rho}(\co(\setA))=
\adjustlimits\lim_{n\to\infty}\inf_{A_{i}\in\co(\setA)}\|A_{n}\cdots A_{1}\|^{1/n}\ge \lim_{n\to\infty}\left(\tfrac{1}{N}\right)^{1/n}
\check{\rho}_{n}(\setA)^{1/n}=\check{\rho}(\setA),
\]
and together with \eqref{E:checkrho} this yields the required equality.
\end{proof}

\section{Concluding remarks}\label{S:conclude}

\subsection{Sets of matrices with independent column uncertainty}

Since the spectral radius does not change during the transposition of a
matrix, then all the assertions of Theorem~\ref{T:Hset} remain to be valid
for the sets of matrices taken from the totality of $\Hset^{\transpose}$-sets
of matrices which is obtained by transposition of matrices from $\Hset$-sets.
In particular, in the course of transposing, the sets of matrices with
independent row uncertainty turn into the so-called \emph{sets of matrices
with independent column uncertainty}~\cite{BN:SIAMJMAA09}.

Note that for the $\Hset^{\transpose}$-sets of matrices the hourglass
alternative, generally speaking, is not valid. In this connection the
question naturally arises about further expansion of classes of matrices for
which the theorems proven in the article hold.

\subsection{Terminology}
We have borrowed the term `set of matrices with independent row (or column)
uncertainty' from the recent work~\cite{BN:SIAMJMAA09}, although such a kind
of sets of matrices in fact have been used for a long time in the theory of
parallel computing and the theory of asynchronous systems.
In~\cite{Prot:MP15} the same sets of matrices got the name \emph{product
families}.

In the special case when each of the rows of the matrix $\setA$ coincides
with the corresponding row of either a predetermined matrix $A$ or the
identity matrix $I$, this type of matrices is sometimes
called~\cite{KKKK:AiT83:7:e} \emph{mixtures} of the matrices $A$ and $I$, see
also a brief survey in~\cite{Koz:ICDEA04}. An example, in which the mixtures
of matrices arise, is the so-called linear asynchronous system
$x_{n+1}=A_{n}x_{n}$, wherein at each time one or more components of the
state vector are updated independently of each other, i.e. each coordinate of
the vector $x_{n+1}$ can take the value of the corresponding coordinates of
$Ax_{n}$ or $x_{n}$.

\section*{Acknowledgments}
The author is indebted to Eugene Asarin for numerous fruitful discussions and
remarks.

The work was funded by the Russian Science Foundation, Project No.
14-50-00150.

\bibliographystyle{elsarticle-num}
\bibliography{Hourglass}

\end{document}